\documentclass{amsart}
\usepackage{amsmath,amssymb,hyperref,mathrsfs,graphicx}
\usepackage[utf8x]{inputenc}
\usepackage{amsmath}
\usepackage{amsfonts}
\usepackage{latexsym}
\usepackage{amsthm}
\usepackage{amssymb,amscd}
\usepackage{xargs}
\usepackage{graphicx}
\usepackage{color}
\usepackage{enumitem}

\newtheorem{thm}{Theorem}[section]

\newtheorem{lem}[thm]{Lemma}

\newtheorem{cor}[thm]{Corollary}
\newtheorem{rmk}[thm]{Remark}
\newtheorem{ex}[thm]{Example}
\newcommand{\be}{\begin{eqnarray}}
\newcommand{\ee}{\end{eqnarray}}
\newcommand{\beal}{\begin{aligned}}
\newcommand{\enal}{\end{aligned}}

\newcommand{\eps}{\varepsilon}

\newcommand{\T}{\mathbb{T}}
\newcommand{\R}{\mathbb{R}}

\newcommand{\Z}{\mathbb{Z}}

\newcommand{\om}{\omega}

\newcommand{\cO}{\mathcal{O}}

	\def\textr{\textcolor{red}}

\title{Suspension of the Billiard maps in the Lazutkin's coordinate}
\author{Jianlu Zhang}
\address{Department of Mathematics, University of Toronto \\Ontario, Canada, m5s2n1}
\email{jianlu.zhang@utoronto.ca}
\thanks{}
\date{}
\begin{document}
\maketitle
\begin{abstract}
In this paper we proved that under the Lazutkin's coordinate, the billiard map can be interpolated by a time-1 flow of a Hamiltonian $H(x,p,t)$ which can be formally expressed by
\[
H(x,p,t)=p^{3/2}+p^{5/2}V(x,p^{1/2},t),\quad(x,p,t)\in\T\times[0,+\infty)\times\T,
\]
where $V(\cdot,\cdot,\cdot)$ is $C^{r-5}$ smooth if the convex billiard boundary is $C^r$ smooth. We also show several applications of this suspension in exploring the properties of the billiard maps.
\end{abstract}

\section{Introduction}
Let's concern the following reflective mechanism: For a convex domain $\Omega$ of $\R^2$ with a sufficiently smooth boundary $\partial\Omega$, there exists a inner partical which makes uniform motion in straight line. When it hits $\partial\Omega$, the impact angle is the same as the reflected angle (see figure \ref{fig1}). The {\bf Billiard map} is just defined by the correspondence of two adjacent reflective points, i.e. $\phi:P_0\rightarrow P_1$.
\begin{figure}
\begin{center}
\includegraphics[width=10cm]{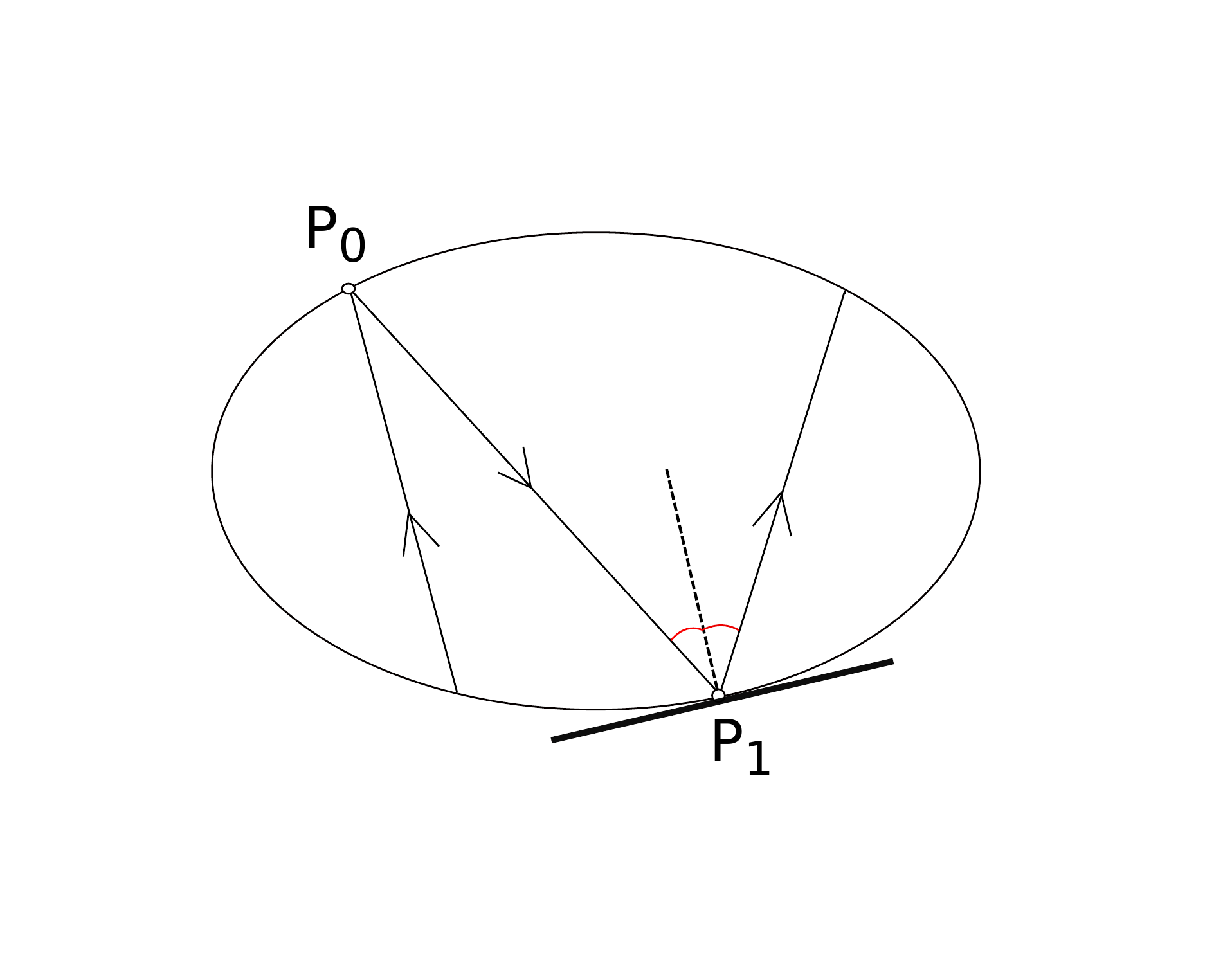}
\caption{ }
\label{fig1}
\end{center}
\end{figure}
According to this mechanism, if there exists a closed convex curve $\Gamma$ lying in $\Omega$, such that any tangent of $\Gamma$ remains a tangent line after once reflection, then we call it a {\bf caustic} (see figure \ref{fig2}). If we formalize the arc length of $\partial\Omega$ by $1$, then each reflective point can be fixed by the coordinate $(s,v)$, and
\be
\phi:(s,v)\rightarrow(s_1,v_1)
\ee
becomes a map on the closed annulus $\mathbb{A}=\{(s,v)|(s,v)\in\T\times[0,\pi]\}$. Here we take $s$ by the arc-parameter and $v$ by the reflected angle. Obviously $\phi$ leaves the boundary of $\mathbb{A}$, $\partial\mathbb{A}:= \{t = 0\} \bigcup\{t = \pi\}$ invariant. Moreover, the billiard map actually belongs to the type of {\bf exact monotone twist maps}, this is because we can find a generating function by 
\be
h(s,s^+):=-d(P_0,P_1), \quad (s,s^+)\in\T^2, 
\ee
where $d(\cdot,\cdot)$ is the Euclid distance of $\R^2$, and
\be
\partial_1h=\cos v,\quad \partial_2h=-\cos v^+.
\ee
The twist property is implied by $-\partial_{12}h>0$ once the boundary is strictly convex \cite{M2}. With this variational approach, a caustic will correspond to an invariant curve in $\mathbb A$. J. Mather gave a short and smart proof on the nonexistence of caustics:
\begin{thm}\cite{Ma}
 If $\partial\Omega$ has a flat point, that is if there is a point, where the curvature vanishes, then the billiard map $\phi$ has no invariant curve.
\end{thm}
\begin{figure}
\begin{center}
\includegraphics[width=10cm]{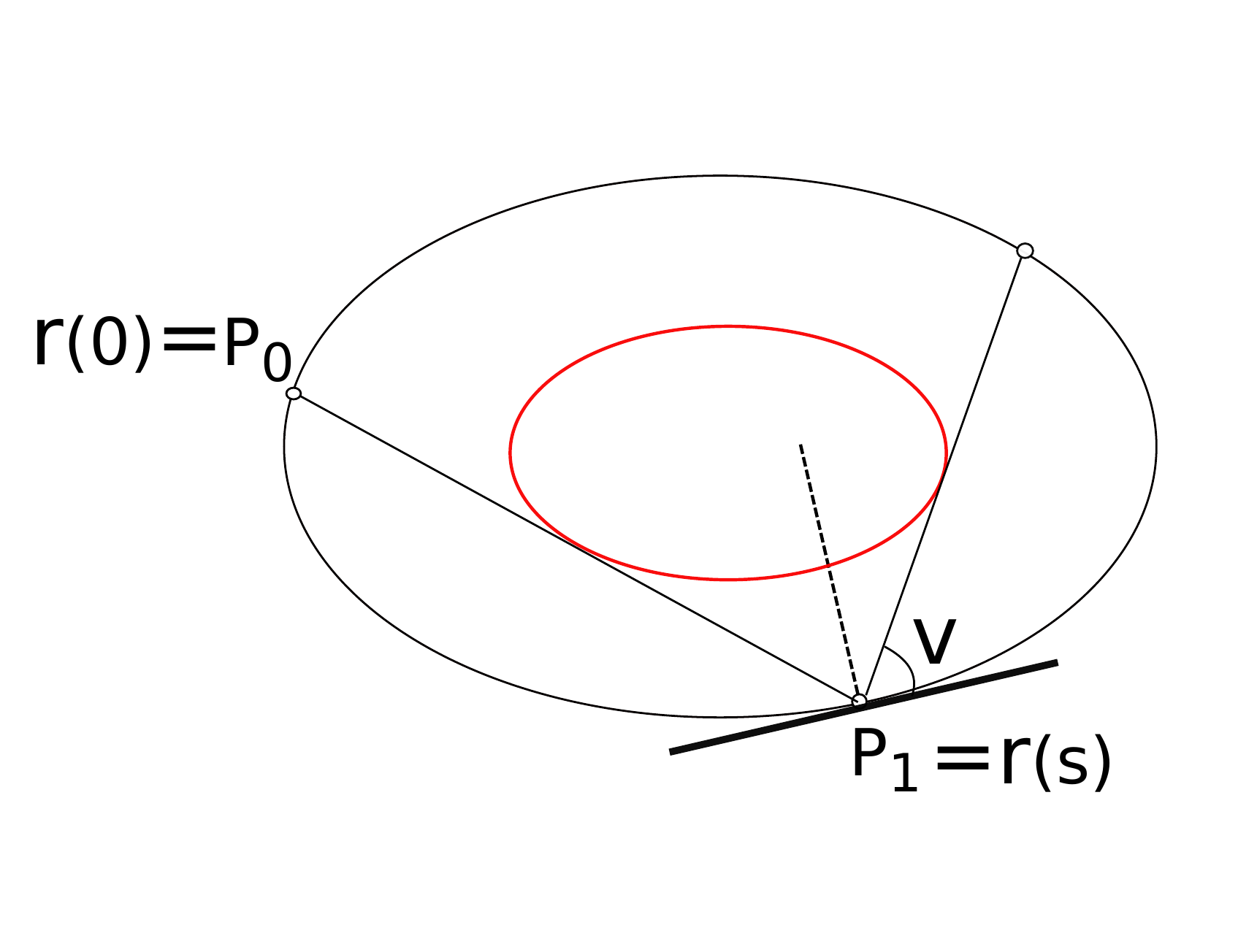}
\caption{ }
\label{fig2}
\end{center}
\end{figure}
On the other side, Lazutkin succeeded in applying the averaging method to prove the existence of plenty of caustics which are close to the boundary $\partial\Omega$:
\begin{thm}\cite{L}
$\forall$ small $0<a<1$, fixed constant $\sigma\geq3$ and $A>0$, we denote by $E(a)$ the set of all the $(A,\sigma)-$Diophantine number in $[0,a]$, i.e. $\forall \eta\in E(a)$, there exists a constant $A$, such that
\[
|n\eta-m|\geq A|m|n^{0.5-\sigma},\quad\forall m\in\Z, n\in\Z_+.
\]
For $C^r$ smooth boundary $\partial\Omega$ with $r\geq553$, there exists a constant $b^*\in(0,a)$ depending only on $\|\partial\Omega\|_{C^r}$, $\sigma$ and $A$, such that all the quasi-periodic curves persist with the frequency in $E(a)\bigcap[0,b^*]$. 
\end{thm} 
\begin{rmk}
Recall that $E(a)$ is a sparse set in $[0,a]$ but has a positive Lebesgue measure:
\[
mes(E(a))\geq a-c_1 a^{\sigma-0.5},
\]
where $c_1$ depends only on $\sigma$ and $A$, and as $\sigma\rightarrow\infty$, $c_1\rightarrow0$.

In another paper \cite{D}, Douady improved the smoothness of boundary by $r\geq 6$.
\end{rmk}
Lazutkin's prove is due to the following procedure: for sufficiently small reflected angle $v$, the billiard map can be expressed by
\begin{eqnarray}\label{eq1}
s^+&=&s+\alpha_1(s)v+\alpha_2(s)v^2+\alpha_3(s)v^3+F(s,v)v^4\nonumber\\
v^+&=&v+\beta_2(s)v^2+\beta_3(s)v^3+G(s,v)v^4,
\end{eqnarray}
where
\be
\alpha_1(s)=2\rho(s),\;\alpha_2(s)=\frac{4}{3}\rho(s)\rho'(s),\nonumber\\
\alpha_3(s)=\frac{2}{3}\rho^2\rho''+\frac{4}{9}\rho\rho'^2,\nonumber\\
\beta_2(s)=-\frac{2}{3}\rho',\;\beta_3(s)=-\frac{2}{3}\rho\rho''+\frac{4}{9}\rho'^2,\nonumber
\ee
with $\rho(s)$ being the curvature of $\partial\Omega$ and $s$ being the arc parameter. Then he found an exact symplectic transformation by 
\begin{eqnarray}
x &=&\frac{\int_0^s\rho(\tau)^{-2/3}d\tau}{\int_0^1\rho(s)^{-2/3}ds},\nonumber\\
y &=&\frac{4\rho(s)^{1/3}\sin v/2}{\int_0^1\rho(s)^{-2/3}ds},
\end{eqnarray}
and then (\ref{eq1}) becomes
\begin{eqnarray}
x^+&=&x+y+y^3f(x,y),\nonumber\\
y^+&=&y+y^4g(x,y),
\end{eqnarray}
with the invariant symplectic form by $dx\wedge d{y^2}/{2}$. If we take $l=y^2/2$, then
\begin{eqnarray}\label{eq}
x^+&=&x+\sqrt{2}l^{1/2}+2\sqrt{2}l^{3/2}f(x,\sqrt{2l}),\\
l^+&=&l+4\sqrt{2}l^{5/2}g(x,\sqrt{2l})\cdot[1+\sqrt{2}l^{3/2}g(x,\sqrt{2l})]\nonumber
\end{eqnarray}
is exactly standard symplectic with the generating function 
\begin{equation}\label{generating}
\tilde{h}(x,x^+)=4C_1^2\int_x^{x^+}\rho^{2/3}(s(\tau))d\tau+4C_1^3h(s,s^+),\quad C_1=\Big{(}\int_0^1\rho^{-2/3}(s)ds\Big{)}^{-1}
\end{equation}
which satisfies 
\[
-\partial_1\tilde{h}(x,x^+)=l,\quad\partial_2\tilde{h}(x,x^+)=l^+.
\]
Formally we can find (\ref{eq}) becomes a nearly integrable map for small $l>0$, that's why several steps KAM iterations can be applied and he can get the persistence of positive measure caustics.\\

As now the Lazutkin's coordinate proves to be a powerful tool in exploring the properties of the billiard maps in the nearing boundary domain, we can naturally make a comparison with the convex nearly integrable Hamiltonian flow. That's because the latter is more flexible and developed in practical applications, e.g. we can get resonant normal forms by the KAM iterations (related with the existence of resonant cuastics), or use the Aubry Mather theory or weak KAM theorem get the properties for all kinds of invariant set. This motivation urges us to prove the following conclusion: 
\begin{thm}[Main Conclusion]
There exists a smooth Hamiltonian $H(x,l,t)$ defined for $(x,l,t)\in\T\times[0,\infty)\times\T$, of which the billard map (\ref{eq}) can be interpolated by the time-1 flow $\phi_H^1$ for $0\leq l\ll1$. Moreover, $H(x,l,t)$ can be formally expressed by
\begin{equation}\label{formal}
H(x,l,t)=\frac{2\sqrt{2}}{3}l^{3/2}+l^{5/2}V(x,\sqrt{l}, t),\quad(x,l,t)\in\T\times[0,+\infty)\times\T.
\end{equation}
\end{thm}
We should confess that J. Moser is the first mathematician making a connection between monotones twist maps and time-1 periodic flow of Hamiltonians, and in \cite{M} he gave a precise prove for how to suspend  general monotone twist maps by time-periodic Hamiltonians. But (\ref{eq}) is just of the type excluded by his paper, because from (\ref{formal}) you can find a singular kinetic part $\frac{2\sqrt{2}}{3}l^{3/2}$, with the $\|\|_{C^2}$ norm blowing up as $l\rightarrow0$. So we need to cover this case by preciser quantitative evaluation.

The heuristic idea of the proof is the following: first, suspend (\ref{eq}) by a straight flow Hamiltonian, which is unnecessarily to be time periodic; then we can evaluate the kinetic part and perturbation part, and strictly separate they two under the $\|\cdot\|_{C^2}$ norm; at last, due to the evaluation we slightly modify the Hamiltonian to be time periodic (see Section \ref{2} for more details).\\

Due to the Legendre transformation, we can get the conjugated Lagrangian of (\ref{formal}) by
\be\label{formal2}
L(x,v,t)=\frac{v^3}{6}+v^5U(x,v,t),\quad (x,v,t)\in\T\times[0,+\infty)\times\T,
\ee
which conforms to the basic setting of \cite{Ma2}. So we can use the variational approach of (\ref{formal2}) to prove the following application:
\begin{cor}\label{cor}
Any two adjacent caustics can be connected by one special billiard reflection plan; in other words, for any two adjacent invariant curves $\Gamma_0$, $\Gamma_1$ of (\ref{formal}) and any two small open neighborhoods $U_0$, $U_1$ of each, there exists one trajectory $\phi_H^t(x^*,l^*)$ passing by these two neighborhoods in turns.
\end{cor}
We postpone the proof of this Corollary to Section \ref{3}, where we list some special applications of the Aubry Mather theory for the billiard maps as well.
\section{Proof of the main Theorem}\label{2}
Recall that (\ref{eq1}) is explicit only for sufficiently small $v$, and the expansion of (\ref{generating}) can be estimated by:
\begin{equation}\label{expansion}
\tilde{h}(x,x^+)=\frac{1}{6}\Delta x^3+\cO(\Delta x^4)
\end{equation}
for sufficiently small $\Delta x=x^+-x$. So we just need to care the dynamic behavior of the Lazutkin's map (\ref{eq}). Then we can slightly change $\tilde{h}(x,X)$ into 
\[
h(x,X)=\frac{1}{6}(X-x)^3+(X-x)^4P(x,X)\rho(X-x)
\]
with $P(x+1,X+1)=P(x,X)$, $x\leq X$ and $x\in\R$. Here $\rho(\cdot)\in C^r(\R,\R)$ satisfying
\[
\rho(t)=
\begin{cases}
1 & t\in[0,\eps]\\
0 & t\in[\sqrt{\eps},\infty)
\end{cases}
\]
is just a smooth transitional smooth function with $\|\rho''(t)\|\leq\frac{1}{\eps}$ and $\|\rho'(t)\|\leq\frac{1}{\sqrt{\eps}}$. We can always take $\eps\ll1$ sufficiently small such that there exists a constant $c$ depending only on $\eps$ and 
\[
-\partial_{12}h(x,X)>c(X-x),\quad\forall x<X.
\]
\begin{rmk}
Actually, this modified generating function conform with (\ref{eq}) only for the domain $\Theta:=\{(x,l)\in\T\times[0,\infty)|\sqrt{2}l^{1/2}+2\sqrt{2}l^{3/2}f(x,\sqrt{2l})\leq\eps\}$. Due to the conclusion of \cite{L}, We can always find a list of KAM tori $\mathcal{T}_{\omega}=\{(x,l_\omega(x))|x\in\T\}$ with the Diophantine frequency $\om\in E(\eps)$. We can always pick one $\mathcal{T}_\om$ which encloses an invariant domain $\Lambda$ with $\mathcal{T}_0=\{(x,0)|x\in\T\}$ as the lower bound. $\Theta\subseteq\Lambda$ and (\ref{eq}) is still the map in $\Lambda$.
\end{rmk}

This proof for this Theorem is outlined as follows: We first suspend (\ref{eq}) by a smooth but non-periodic Hamiltonian $H(x,l,t)$, then slightly modified it into a periodic but only piecewise continuous $\hat{H}(x,l,t)$. The last step, we polish $\hat{H}(x,l,t)$ to $\tilde{H}(x,l,t)$ which becomes smooth of time $t$ again, $t\in\T$. Due to (\ref{eq}), $\tilde{H}(x,l,t)$ can be formally expressed as (\ref{formal}).
\subsection{Suspension}

\begin{lem}\label{suspension}
There exists a Lagrangian $L(x,v,t)$ such that 
\begin{equation}\label{variation}
h(x,X)=\inf_{\begin{subarray}{1}
\gamma\in Lip([0,1],\R)\\\gamma(0)=x,\gamma(1)=X
\end{subarray}
}\int_0^1L(\gamma(t),\dot\gamma(t),t)dt, \quad X\geq x
\end{equation}
for $h(x,X)=c_1(X-x)^3+\cO(|X-x|^4)$, $c_1>0$.
\end{lem}
\begin{proof}
Actually, the selection of $L(x,v,t)$ is rather arbitrary, so we can specially choose the one with linear Euler-Lagrange equation, i.e.
\[
L_{vx}(x,v,t)v+L_{vt}(x,v,t)=L_x(x,v,t).
\]
Differentiate both sides with respect to $v$ variable, we should have
\[
L_{vvx}(x,v,t)v+L_{vvt}(x,v,t)=0.
\]
By characteristic method we can solve previous P.D.E by
\[
G(x-vt,v)=L_{vv}(x,v,t),
\]
where $G(\cdot,\cdot)$ is a designated function later on. Then
\begin{eqnarray*}
L(x,v,t)&=&L(x,0,t)+\int_0^vL_v(x,\eta,t)d\eta\\
&=&L(x,0,t)+\int_0^vL_v(x,v-\eta,t)d\eta\\
&=&L(x,0,t)+L_v(x,0,t)v+\int_0^v\eta L_{vv}(x,v-\eta,t)d\eta\\
&=&L(x,0,t)+L_v(x,0,t)v+\int_0^v\eta G(x-(v-\eta)t,v-\eta)d\eta\\
&=&L(x,0,t)+L_v(x,0,t)v+\int_0^v(v-\eta)G(x-\eta t,\eta)d\eta.
\end{eqnarray*}
This is just a formal deduction, and we can specially choose the boundary conditions by $L(x,0,t)=L_v(x,0,t)=0$. Conversely, these trivial boundary conditions constraint
\begin{equation}\label{twist}
-\partial_{12}h(x,X)=G(x,X-x)
\end{equation}
if we take it back into (\ref{variation}). By aware that $-\partial_{12}h(x,X)>0$ only for $X>x$, and $G(x,0)=0$. That means the twist index decays to $0$ for $X-x\rightarrow 0$, which is quite different from the case considered by J. Moser in \cite{M}. 
Finally we can solve the Lagrangian by
\[
L(x,v,t)=-\int_0^v(v-\eta)\partial_{12}h(x-\eta t,x+\eta(1-t))d\eta
\]
for $(x,v,t)\in\T\times[0,+\infty)\times[0,1]$. 
\end{proof}
By the Legendre transformation we get 
\[
H(x,l,t)=\max_v\{vl-L(x,v,t)\},\quad (x,l,t)\in\T\times[0,+\infty)\times[0,1],
\]
where the maximum achieves for $L_v(x,v,t)=l$. Recall that 
\begin{equation}\label{momentum1}
L_v(x,v,t)=-\int_0^v\partial_{12}h(x-\eta t,x+\eta(1-t))d\eta\geq0
\end{equation}
and
\[
L_{vv}(x,v,t)=-\partial_{12}h(x-v t,x+v(1-t))\geq0
\]
with `=' holds only for $v=0$, so $L_v$ is a diffeomorphism between $v$ and $p$.
\begin{rmk}
We can generalize (\ref{variation}) to a rescaled version: By taking
\begin{equation}\label{rescale}
L_{ab}(x,v,t)=\frac{1}{b-a}L(x,v,\frac{t-a}{b-a}),
\end{equation}
the following variational principle holds:
\begin{equation}
h(x,X)=\inf_{\begin{subarray}{1}
\gamma\in Lip([a,b],\R)\\\gamma(a)=x,\gamma(b)=X
\end{subarray}
}\int_a^bL_{ab}(\gamma(t),\dot\gamma(t),t)dt, \quad X\geq x.
\end{equation}
Accordingly, 
\[
H_{ab}(x,l,t)=\frac{1}{b-a}H(x,l,\frac{t-a}{b-a}),\quad (x,l,t)\in\T\times[0,+\infty)\times[0,1].
\]
\end{rmk}
\subsection{Periodic Extension}Mention that $H(x,l,t)$ is not periodic of $t$ yet, so we need to modify it to adapt $\phi$.
Recall that $X=f(x,l)$ and $-\partial_{12}h(x,X)=\frac{\partial l}{\partial X}\geq0$ for $X\geq x$, so we can assume 
\[
S_\phi(x,l):=h(x,f(x,l)),\quad(x,l)\in\R\times[0,\infty)
\]
and easily prove that
\[
dS_\phi=LdX-ldx=g(x,l)df(x,l)-ldx,\quad(x,l)\in\R\times(0,\infty). 
\]
 Now for a integrable map $\psi^\kappa:(x,l)\rightarrow (x+\kappa \sqrt{2l},l)$ on $\T\times[0,\infty)$, we can similarly
get
\[
S_{\psi^\kappa}(x,l)=\frac{\sqrt{2}\kappa}{3}l^{3/2},\quad(x,l)\in\R\times[0,\infty).
\]
For sufficiently small $0<\kappa<\frac{1}{5}$, we can define a modified map 
\begin{equation}\label{H'}
\varphi(x,l):=\psi^{-\kappa}\circ\phi\circ\psi^{-\kappa}(x,l)=(f'(x,l),g'(x,l)),\quad(x,l)\in\R\times[0,\infty)
\end{equation}
satisfying
\begin{eqnarray}
f'(x,l)&=&x+\sqrt{2l}(1-2\kappa)+2\sqrt{2}l^{3/2}\tilde{f}'(x-\kappa\sqrt{2l},\sqrt{2l})\nonumber\\
g'(x,l)&=&l+4\sqrt{2}l^{5/2}\tilde{g}'(x-\kappa\sqrt{2l},\sqrt{2l})
\end{eqnarray}
which still ensures the strictly twist property $\partial_2 f'(x,l)>0$. Moreover, we can find $S_{\varphi}(x,l)$
satisfying 
\begin{equation}\label{extension}
dS_\varphi=g'(x,l)df'(x,l)-ldx,\quad(x,l)\in\R\times(0,\infty).
\end{equation}
This is because the following Lemma:
\begin{lem}[Ex 57.5 in \cite{G}]
Suppose $F$, $G$ are two exact symplectic maps of $(T^*M,dy\wedge dx)$, if $G^*ydx- ydx= dS_G$ and $F^*ydx-ydx= dS_F$  then
\[
(F\circ G)^*ydx-ydx=d [S_F\circ G+S_G].
\]
\end{lem}
We can apply this Lemma with $M=\T$ and restrict on the upper semi-part $\{(x,l)\in T^*\T|l>0\}$. Once we get $S_{\varphi}(x,l)$, we can get the generating function by
\[
h_\varphi(x,X)=S_\varphi(x,f'^{-1}(x,X)),\quad X>x,\;x\in\R
\]
because $X=f'(x,l)$ and the twist property. Actually, there exists a constant $c'>0$ depending only on $\kappa$ and $\eps$, such that $\partial_2f'(x,l)\geq c'/\sqrt{l}$. Then $0<\partial_2f'^{-1}(x,X)\leq\sqrt{f'^{-1}(x,X)}/c_1$ and 
\begin{eqnarray*}
-\partial_1h_\varphi(x,X)&=&f'^{-1}(x,X),\\
\partial_2h_\varphi(x,X)&=&g'(x,f'^{-1}(x,X)),\\
-\partial_{11}h_\varphi(x,X)&=&-\frac{\partial_1f'(x,f'^{-1}(x,X))}{\partial_2 f'(x,f'^{-1}(x,X))},\\
-\partial_{12}h_\varphi(x,X)&=&\partial_2f'^{-1}(x,X),\\
\partial_{22}h_\varphi(x,X)&=&\partial_2g'(x,f'^{-1}(x,X))\cdot\partial_2f'^{-1}(x,X),
\end{eqnarray*}
all converge to $0$ as $X\rightarrow x$. That means $h_{\varphi}(x,X)$ can be at least $C^2-$smoothly extended to the domain $\{X\geq x|x\in\R\}$.\\

As we have already got the generating function $h_{\varphi}(x,X)$, we can apply Lemma \ref{suspension} one more time and get a modified Hamiltonian $H'(x,l,t)$ such that $\varphi_{H'}^t$ is the interpolating flow
with $\varphi_{H'}^1=\varphi$ and $\varphi_{H'}^0=id$, $t\in[0,1]$, $(x,l)\in\T\times[0,\infty)$. In other words, $\phi$ can be suspended by the following modified flow
\begin{equation}
\chi_t=\left\{
\begin{array}{cccccc}
\psi^t & & 0\leq t<\kappa,\\
\varphi_{H'}^{\frac{t-\kappa}{1-2\kappa}}\circ\psi^\kappa & & \kappa\leq t\leq1-\kappa,\\
\psi^{t-1}\circ\varphi & & 1-\kappa<t\leq 1,
\end{array}\right.
\end{equation}
which is generated by the periodic Hamiltonian 
\begin{equation}\label{piecewise}
\hat{H}(x,l,t)=\left\{
\begin{array}{cccccc}
\frac{2\sqrt{2}}{3}l^{3/2} & & t\in[0,\kappa)\bigcup(1-\kappa,1],\vspace{5pt}\\
\frac{1}{1-2\kappa}H'(x,l,\frac{t-\kappa}{1-2\kappa})& & \kappa\leq t\leq 1-\kappa,
\end{array}\right.
\end{equation}
with $(x,l,t)\in\T\times[0,\infty)\times\T$.  Later we will see that for sufficiently small $l\ll1$, 
\[
\partial_{ll}H'(x,l,t)\geq c''/\sqrt{l}
\] 
with $c''\sim\cO(1)$ a constant depending only on $\eps$ (see the part of {\it Re-evaluation}). 

Unfortunately, there comes out two discontinuities of $\hat{H}(x,l,t)$ at $t=\kappa$, $1-\kappa$. Later we will polish $\hat{H}(x,l,t)$ into a smooth $\tilde{H}(x,l,t)$ with flow $\tilde{\varphi}^t$ satisfying $\tilde{\varphi}^0=id$ and $\tilde{\varphi}^1=\phi$.
\subsection{Smoothness}
To find a smooth time-periodic $\tilde{H}(x,l,t)$ being the suspension for $\phi$, the following tools are necessary:
\subsubsection{Re-evaluation of $H'(x,l,t)$}
The analysis of this subsection is based on a simple fact: the Euler-Lagrange flow of Lemma \ref{suspension} has a constant velocity, i.e.
\begin{eqnarray}\label{speed}
\partial_lH'(x(t),l(t),t)&=&\partial_lH'(x+(X-x)t,l(t),t)\nonumber\\
&=&v(t)\nonumber\\
&=&X-x\nonumber\\
&=&\sqrt{2l}(1-2\kappa)+2\sqrt{2}l^{3/2}f(x-\kappa\sqrt{2l},\sqrt{2l}),\quad\forall t\in[0,1],
\end{eqnarray}
if $x(0)=x$, $x(1)=X$ and $l(0)=l$. This is because $\varphi$ has been established in (\ref{H'}). \\

On the other side, the corresponding generating function $h_{\varphi}$ satisfies
\begin{eqnarray}\label{eva h'}
-\partial_{12}h_{\varphi}(x-\eta t,x-\eta t+\eta)&=&\Big{(}\frac{\sqrt{2}}{2}(1-2\kappa)l_\eta^{-1/2}+3\sqrt{2l_\eta}\tilde{f}'(x-\eta t-\kappa\sqrt{2l_\eta},\sqrt{2l_\eta})\Big{)}^{-1}\nonumber\\
&=&\frac{\sqrt{2l_\eta}}{1-2\kappa}-\frac{6\sqrt{2}}{(1-2\kappa)^2}l_\eta^{3/2}f(x-\eta t-\kappa\sqrt{2l_\eta},\sqrt{2l_\eta})+\cO(l_\eta^{5/2})
\end{eqnarray}
with $\|\tilde{f}'(\cdot,\cdot)\|_{C^{r-5}}$ bounded and
\be
\eta=\sqrt{2l_\eta}(1-2\kappa)+2\sqrt{2}l_\eta^{3/2}f(x-\eta t-\kappa\sqrt{2l_\eta},\sqrt{2l_\eta})
\ee
as long as $0<\eta\leq X-x\ll1$. Conversely,
\be
l_\eta=\frac{\eta^2}{2(1-2\kappa)^2}-\frac{\eta^4}{(1-2\kappa)^5}f(x-\eta t-\frac{\kappa\eta}{1-2\kappa},\frac{\eta}{1-2\kappa})+\cO(\eta^5).
\ee
Taking (\ref{eva h'}) into (\ref{momentum1}) implies that $\forall (x,l,t)\in\T\times[0,\eps)\times\T$,
\begin{eqnarray}\label{momentum2}
l(t)&=&L'_v(x(t),v(t),t)\nonumber\\
&=&\int_0^{X-x}\frac{\eta}{(1-2\kappa)^2}-\frac{8\sqrt{2}}{(1-2\kappa)^2}l_\eta^{3/2}f(x-\eta t-\kappa\sqrt{2l_\eta},\sqrt{2l_\eta})d\eta\nonumber\\
&=&\frac{(X-x)^2}{2(1-2\kappa)^2}-\int_0^{X-x}\frac{8\sqrt{2}}{(1-2\kappa)^2}l_\eta^{3/2}f(x-\eta t-\kappa\sqrt{2l_\eta},\sqrt{2l_\eta})d\eta\nonumber\\
&=&l+\frac{4l^2}{1-2\kappa}f(x-\kappa\sqrt{2l},\sqrt{2l})-\int_0^{X-x}\frac{8\sqrt{2}}{(1-2\kappa)^2}\cdot\frac{\eta^3f(x,0)}{2\sqrt{2}(1-2\kappa)^3}+\cO(\eta^4)d\eta\nonumber\\
&=&l+l^{5/2}\bar{f}'(x,\sqrt{2l},t)
\end{eqnarray}
with $\|\bar{f}'(\cdot,\cdot,\cdot)\|_{C^{r-6}}$ is bounded due to L'Hospitale Principle for $0<X-x\ll1$. Here $L'$ is the conjugated Lagrangian of $H'$. Now if we take (\ref{momentum2}) back into (\ref{speed}), we finally re-evaluate $H'(x,l,t)$ by
\be
H'(x,l,t)=\frac{2\sqrt{2}}{3}(1-2\kappa)l^{3/2}+l^{5/2}V'(x,\sqrt{l},t),\quad(x,l,t)\in\T\times\R^+\times\T,
\ee
where $\|V'(\cdot,\cdot,\cdot)\|_{C^{r-5}}$ is bounded.
\begin{rmk}
Recall that from (\ref{rescale}) of previous remark, (\ref{piecewise}) now becomes
\be\label{square hamilton}
\hat{H}(x,l,t)=\left\{
\begin{array}{cccccc}
\frac{2\sqrt{2}}{3}l^{3/2} & & t\in[0,\kappa)\bigcup(1-\kappa,1],\vspace{5pt}\\
\frac{2\sqrt{2}}{3}l^{3/2} +\frac{l^{5/2}}{1-2\kappa}V'(x,\sqrt{l},\frac{t-\kappa}{1-2\kappa})& & \kappa\leq t\leq 1-\kappa.
\end{array}\right.
\ee
We can see that the kinetic part is always $\frac{2\sqrt{2}}{3}l^{3/2}$ for all $t\in[0,1]$, which is much greater than the perturbation part $\frac{l^{5/2}}{1-2\kappa}V'(x,\sqrt{l},\frac{t-\kappa}{1-2\kappa})$ !!!!\\
\end{rmk}

\subsubsection{New generating function}
 Here we involve a new type generating function $S(X,l,t):=Xl+w(X,l,t)$ which corresponds to the flow map $\phi_H^t$ of the Hamiltonian $H(x,l,t)$, where $(x,l,t)\in\T\times[0,\infty)\times[0,1]$ ($H(x,l,t)$ is not necessarily time-periodic). Recall that $\phi_H^t$ is exact symplectic, that means 
\be\label{general}
\partial_Xw=L-l,\;\partial_lw=x-X
\ee
for $\phi_H^t(x,l)=(X,L)$. 
{ Then we have
\[
-\partial_lw(X,l,t)=X-x=\int_0^t\partial_2H(\phi_H^s(x,l),s)ds.
\]
By deriving of variable $t$ on both sides, 
\[
-\partial_{lt}w(X,l,t)=\partial_2H(\phi_H^t(x,l),t)=\partial_2H(X,L,t)
\]
and by integrating both sides we get
\begin{eqnarray}\label{connection}
H(X,L,t)-H(X,0,t)&=&H(X,L,t)\nonumber\\
&=&\int_0^L\partial_2H(X,Z,t)dZ\nonumber\\
&=&-\int_0^l\partial_{lt}w(X,\zeta,t)\frac{\partial Z}{\partial\zeta}d\zeta,\quad(\phi_H^t(\xi,\zeta)=(X,Z))\nonumber\\
&=&-\int_0^l\partial_{lt}w(X,\zeta,t)(1+\partial_{lX}w(X,\zeta,t))d\zeta.
\end{eqnarray}
Then  $\partial_{22}H(X,L,t)>0$ is equivalent to
\be\label{positive-defn}
\frac{\partial_{tll}w(X,l,t)}{1+\partial_{Xl}w(X,l,t)}<0.
\ee
}
Now we turn back to the Hamiltonian $\hat{H}(x,l,t)$ which conforms to (\ref{square hamilton}). We can polish it with the following convolution trick:\\

For $0<\rho<\kappa$, we can define a bump function by $\eta(s)\in C^\infty(\R)$, $0\leq\eta\leq 1$, $\eta(s)=0$ for $|s|\geq1$ and
\[
\int_\R\eta(s)ds=1.
\]
Then
\be
H^*(x,l,t):=\hat{H}*\eta_\rho,\quad\eta_\rho:=\frac{\eta(t/\rho)}{\rho}
\ee
becomes $C^{r-5}$ smooth on $(x,l,t)\in\T\times[0,\eps)\times[0,1/2]$. Moreover, 
\[
H^*(x,l,t)=\frac{2\sqrt{2}}{3}l^{3/2}=\hat{H}(x,l,t)
\] 
for $0\leq t<\kappa-\rho$ and formally 
\be\label{smooth hamilton}
H^*(x,l,t)=\frac{2\sqrt{2}}{3}l^{3/2}+\frac{l^{5/2}}{1-2\kappa}V^*(x,\sqrt{l},t)
\ee 
for $t\in[0,1/2]$. Actually,
\be
V^*(x,\sqrt{l},t)=\int_{-\infty}^{+\infty}\hat{V}(x,\sqrt{l},t-s)\cdot\eta_\rho(s)ds
\ee
with
\be
\hat{V}(x,\sqrt{l},t)=\left\{
\begin{array}{cccccc}
0 & & t\in[0,\kappa)\bigcup(1-\kappa,1],\vspace{5pt}\\
V'(x,\sqrt{l},\frac{t-\kappa}{1-2\kappa})& & \kappa\leq t\leq 1-\kappa.
\end{array}\right.
\ee
So $\|\partial_1^{\alpha_1}\partial_2^{\alpha_2}(V^*(x,\sqrt{l},t)-\hat{V}(x,\sqrt{l},t))\|\leq c_\alpha\rho$ for $\alpha=(\alpha_1,\alpha_2)\in\mathbb{N}^2$ with $|\alpha_1|+|\alpha_2|\leq r-5$ and $t\in[0,\kappa-\rho)\cup(\kappa+\rho,1/2]$, where $c_\alpha$ is a constant depending on $\|V'(\cdot,\cdot,\cdot)\|_{C^{r-5}}$.\\

Corresponding to the flow map $\phi_{\hat{H}}^t$ and $\phi_{H^*}^t$, we can separately find $\hat{w}$ and $w^*$.  Such a generating function is always available for $t\in[0,2\kappa]$ with $0<\kappa\ll1$ sufficiently small. Benefit from the special form of (\ref{square hamilton}) and (\ref{smooth hamilton}), for $0<l\ll\rho<\kappa<1/5$, 
\be\label{sp1}
\hat{w}(X,l,t)=-\frac{2\sqrt{2}}{3}l^{3/2}t-l^{5/2}\hat{W}(X,\sqrt{l},t)t
\ee
and 
\be\label{sp2}
w^*(X,l,t)=-\frac{2\sqrt{2}}{3}l^{3/2}t-l^{5/2}W^*(X,\sqrt{l},t)t
\ee
due to (\ref{general}) and (\ref{connection}). Besides, 
\[
\|\partial_1^{\alpha_1}\partial_2^{\alpha_2}\hat{W}(x,\sqrt{l},t)\|\leq \hat{c}_\alpha,\quad\|\partial_1^{\alpha_1}\partial_2^{\alpha_2}W^*(x,\sqrt{l},t)\|\leq c_\alpha^*
\] 
with $\alpha=(\alpha_1,\alpha_2)\in\mathbb{N}^2$ and $\hat{c}_\alpha$, $c_\alpha^*$ depend on $\|V'(\cdot,\cdot,\cdot)\|_{C^{r-5}}$. Recall that $\hat{w}=w^*$ for $0<t\leq\kappa-\rho$, and both $\hat{w}$  and $w^*$ is $C^{r-5}$ smooth on $t\in[0,\kappa)\bigcup(\kappa,1/2]$!!\\


For $t_1$, $t_2$ satisfying $\kappa<t_1<t_2<2\kappa$, we can define a cut-off function $\xi(t)\in C^\infty(\R)$ with $\xi=1$ for $t<t_1$ and $\xi=0$ for $t>t_2$ and
\be\label{construction}
\tilde{w}(X,l,t)=(1-\xi)\hat{w}+\xi w^*
\ee 
becomes a smooth generating function on $t\in[0,2\kappa]$ and 
\[
\tilde{w}(X,l,t)=\hat{w}(X,l,t), \quad t\in[0,\kappa-\rho)\bigcup(t_2,2\kappa].
\]
Actually, we can formally express $\tilde{w}$ by
\be\label{smooth gene}
\tilde{w}(X,l,t)=-\frac{2\sqrt{2}}{3}l^{3/2}t-l^{5/2}t\Big{[}(1-\xi(t))\hat{W}(X,\sqrt{l},t)+\xi (t)W^*(X,\sqrt{l},t)\Big{]}
\ee
Then due to (\ref{connection}) we can find a $C^{r-5}$ smooth Hamiltonian $\tilde{H}(x,l,t)$ and the flow map $\phi_{\tilde{H}}^t(x,l)$ satisfies
\[
\phi_{\tilde{H}}^{1/2}(x,l)=\phi_{\hat{H}}^{1/2}(x,l).
\]
The only thing we need to do is to prove the positive definiteness of $\tilde{H}$, i.e. $\tilde{H}_{ll}>0$ for $t\in[0,1/2]$. From (\ref{construction}) and (\ref{positive-defn}) we just need to prove
\[
\frac{\xi'(t)(w^*_{ll}-\hat{w}_{ll})+\tilde{w}_{tll}}{1+\tilde{w}_{Xl}}<0
\]
for $t\in[t_1,t_2]$.
This is true because
\begin{eqnarray}
\frac{\xi'(t)(w^*_{ll}-\hat{w}_{ll})+\tilde{w}_{tll}}{1+\tilde{w}_{Xl}}&=&\frac{-\frac{1}{\sqrt{2l}}+\cO(l^{1/2})+\cO(t\xi'(t)l^{1/2})}{1+\cO(tl^{3/2})}\nonumber\\
&\leq&-\frac{1}{\sqrt{2l}}+\cO(tl)+\cO(\sqrt l)+\cO(t\xi'(t)\sqrt l)\nonumber\\
&<&-\frac{1}{\sqrt {2l}}+\cO(2\kappa\cdot\frac{4}{\kappa}\sqrt l)\nonumber\\
&<&-\frac{1}{2\sqrt {l}}<0,
\end{eqnarray}
as long as we take $t_2-t_1>\kappa/4$, $0<\kappa<1/5$ and then $\eps$ uniformly small. \\

In this way we remove the discontinuity at $\{t=\kappa\}$, and by the same argument we can remove the discontinuity at $\{t=1-\kappa\}$ and get a totally smooth Hamiltonian $\tilde{H}(x,l,t)$ which is periodic-1 of time $t\in[0,1]$ and $\tilde{H}_{ll}>0$ for $(x,t)\in\T^2$ and $0<l\ll1$. Moreover, from aforementioned argument and (\ref{smooth gene}), $\tilde{H}$ can be finally established by (\ref{formal}).

\section{Aubry Mather theory for the Billiard maps and applications}\label{3}
Let's first make a brief introduction of some elementary definitions and theorems. We concerns the following $C^2-$smooth Tonelli Lagrangian $L(x,v,t)$ with $(x,v,t)\in TM\times\T$, which satisfies these assumptions \cite{Ma2}:
\begin{itemize}
\item {\bf positively definite} the Hessian matrix $L_{vv}$ is positively definite for any $(x,v,t)\in TM\times\T$;
\item {\bf super linear} $L(x,v,t)/\|v\|\rightarrow+\infty$, as $\|v\|\rightarrow+\infty$ for any $(x,t)\in M\times\T$;
\item {\bf completeness} the Euler Lagrange equation of $L(x,v,t)$ is well defined for the whole time $t\in\R$;
\end{itemize}
Based on these, we can define the $\alpha(c): H^1(M,\R)\rightarrow\R$ by
\be\label{alpha}
\alpha(c)=-\inf_{\mu\in\mathfrak{M}_L}\int L-\eta\; d\mu, \quad[\eta]=c
\ee
where $\mathfrak{M}_L$ is the set of all the flow-invariant probability measures on $TM\times\T$. Also we can get its conjugated $\beta(h): H_1(M,\R)\rightarrow\R$ by
\be
\beta(h)=\inf_{\mu\in\mathfrak{M}_L, \rho(\mu)=h}\int L\;d\mu
\ee
as
\[
\langle [\lambda], \rho(\mu)\rangle=\int\lambda\;d\mu,\quad\forall \text{\;closed 1-form \;}\lambda \text{\;on\;} M.
\]
Due to the positive definiteness and super linearity, both of these two functions are convex and superlinear, and
\[
\langle c,h\rangle\leq\alpha(c)+\beta(h),\quad\forall c\in H^1(M,\R),\; h\in H_1(M,\R),
\]
where the equality holds only for $c\in D^+\beta(h)$ and $h\in D^+\alpha(c)$ (sub-derivative set). We denote by $\widetilde{\mathcal{M}}(c)\subset TM\times\T$ the closure of the union for all the supports of the minimizng measures of (\ref{alpha}), which is the so called {\bf Mather set}. Its projection to $M\times\T$ is the {\bf projected Mather set} $\mathcal{M}(c)$. 
From \cite{Ma2} we know that $\pi^{-1}\big{|}_{\mathcal{M}(c)}:M\times\mathbb{S}^1\rightarrow TM\times\mathbb{S}^1$ is a Lipschitz graph, where $\pi$ is the standard projection from $TM\times\T$ to $M\times\T$. 
\begin{rmk}
Actually, aformentioned definition can be applied for the set of closed probability measures $\mathfrak{M}_c$, instead of $\mathfrak{M}_L$. We can get a closed probability measure from a closed loop of $M$ due to the Birkhoff ergodic theorem:
\[
\int f d\mu_c:=\frac{1}{T_c}\int_0^{T_c}f(\gamma_c,\dot{\gamma}_c,t)\;dt,\quad\forall f\in C^{ac}(TM\times\T,\R)
\]
where $T_c$ is the periodic of the loop $\gamma_c$, and $\dot\gamma^-(T_c)\neq\dot\gamma^+(0)$ may be the case. This point is firstly proposed by Ma\~{n}\'e in \cite{Mn}, and we can still get the same $\alpha(c)$ and $\beta(h)$ with this neww setting.
\end{rmk}

Follow the setting of \cite{B}, we define
\begin{equation}
A_c(\gamma)\big{|}_{[t,t']}=\int_t^{t'}L(\gamma(t),\dot{\gamma}(t),t)-\langle\eta_c(\gamma(t)),\dot{\gamma}(t)\rangle dt+\alpha(c)(t'-t),
\end{equation}
\begin{equation}
h_c((x,t),(y,t'))=\inf_{\substack{\xi\in C^{ac}([t,t'],M)\\
\xi(t)=x\\
\xi(t')=y}}A_c(\xi)\big{|}_{[t,t']},
\end{equation}
where $t,t'\in\mathbb{R}$ with $t<t'$, and
\begin{equation}
F_c((x,\tau),(y,\tau'))=\inf_{\substack{\tau=t\mod1\\
\tau'=t'\mod1}}h_c((x,t),(y,t')),
\end{equation}
where $\tau,\tau'\in\mathbb{S}^1$. Then a curve $\gamma:\mathbb{R}\rightarrow M$ is called {\bf c-semi static} if 
\[
F_c((x,\tau),(y,\tau'))=A_c(\gamma)\big{|}_{[t,t']},
\]
for all $t,t'\in\mathbb{R}$ and $\tau=t\mod1$, $\tau'=t'\mod1$. A semi static curve $\gamma$ is called {\bf c-static} if
\[
A_c(\gamma)\big{|}_{[t,t']}+F_c((\gamma(t'),t'),(\gamma(t),t))=0,\quad\forall t,t'\in\mathbb{R}.
\]
The {\bf Ma\~{n}\'e set} $\widetilde{\mathcal{N}}(c)\subset TM\times\T$ is denoted by the set of all the c-semi static orbits, and the {\bf Aubry set} $\tilde{\mathcal{A}}(c)$ is the set of all the c-static orbits. From \cite{B} we can see that $\pi^{-1}:\mathcal{A}(c)\rightarrow \tilde{\mathcal{A}}(c)$ is also a Lipschitz graph. Beside, we know
\[
\widetilde{\mathcal{M}}(c)\subset\tilde{\mathcal{A}}(c)\subset\widetilde{\mathcal{N}}(c).
\]
\vspace{10pt}

Before we apply the Aubry Mather theory to (\ref{formal}), let's first expand the Hamiltonian to the whole cotangent space, i.e.
\be
H(x,l,t)=\frac{2\sqrt{2}}{3}|l|^{3/2}+|l|^{5/2}V(x,\sqrt{|l|}, t),\quad(x,l,t)\in T^*\T\times\T.
\ee
Correspondingly, the Lagrangian becomes symmetric as well:
\be
L(x,v,t)=\frac{|v|^3}{6}+|v|^5U(x,|v|,t),\quad (x,v,t)\in T\T\times\T.
\ee
But we should keep in mind that $\{x\in\T,l=0, t\in\T\}$ forms a rigid `wall' which separates the phase space $T^*\T\times\T$ into two disconnected parts, which are mirror images of each other. The reason we did so is to ensure the globally super-linear and positively definite, as we know $V(x,\sqrt{|l|}, t)=0$ for $|l|>\epsilon$ from the setting of Section \ref{2}. And the completeness is also followed because the velocity $v$ is dominated in a compact region.
\begin{rmk}
The only difference in our case is that $L_{vv}(x,0,t)=0$. Later you can see this point will cause a degeneracy of $\beta(h)$ at $\{h=0\}$.
\end{rmk}
\vspace{10pt}

\begin{lem}\label{lem}
Suppose $\Gamma_0$ and $\Gamma_1$ are two adjacent invariant curves with the rotational number by $w_0$, $w_1$, then for any $w\in(w_0,w_1)$, $\mathcal{N}(c)\bigcap\{t=0\}\varsubsetneqq\T$, where $c\in D^+\beta(w)$.
\end{lem}
\begin{proof}
This is a direct corollary from \cite{Ma3}. As the configuration space $M=\T$, so $\alpha(c)$ is actually $C^1$ smooth. There would be a unique $w=\alpha'(c)$ corresponding to the rotational number of $\widetilde{\mathcal{M}}(c)$. If $w$ is irrational, then $\widetilde{\mathcal{M}}(c)$ contains only one uniquely ergodic measure, so $\tilde{\mathcal{A}}(c)=\widetilde{\mathcal{N}}(c)$ due to \cite{B}. Recall that $\tilde{\mathcal{A}}(c)$ has graph property, then ${\mathcal{N}}(c)\bigcap\{t=0\}\varsubsetneqq\T$. Otherwise, $\widetilde{\mathcal{N}}(c)$ forms an invariant curve which separates $\Gamma_0$ and $\Gamma_1$. This leads to a contradiction.

If $w=p/q$ is rational, there exists a maximal flat $[c_1,c_2]$ of $\alpha(c)$, such that $c\in[c_1,c_2]$ (it may happen that $[c_1,c_2]$ collapse to a single point $c$). If $c\in(c_1,c_2)$, $\widetilde{\mathcal{N}}(c)$ will has a unified rotational number $p/q$ (irreducible fraction); If $c=c_1$ (or $c_2$), then $\widetilde{\mathcal{N}}(c)$ contains $p/q$ periodic orbits and $p/q^-$ heteroclinic orbits (or $p/q^+$ heteroclinic orbits) \cite{Ma3}. This order-preserving structure prevents $\mathcal{N}(c)\bigcap\{t=0\}=\T$ just like before.\vspace{10pt}
 \end{proof}
{\it Proof of Corollary \ref{cor}\;} Now if we define $\Upsilon:=\{c\in H^1(\T,\R)|c\in D^+\beta(w), w\in(w_0,w_1)\}$, then $\Upsilon$ is actually an open interval $(c_-,c_+)$ due to the strictly convex of $\beta(h)$. Moreover, $\tilde{\mathcal{A}}(c_-)\subset\Gamma_0$ and $\tilde{\mathcal{A}}(c_+)\subset\Gamma_1$. So we can claim the following conclusion:
\begin{lem}
$\forall c_1$, $c_2\in\Upsilon$, they are $C-$equivalent in the following sense: there exists a continuous curve $c(s):[0,1]\rightarrow H^1(\T,\R)$ with $c(0)=c_1$ and $c(1)=c_2$, such that $\forall s_0\in[0,1]$, there exists $\delta>0$ which ensures
\[
\langle c(s)-c(s_0), V(\mathcal{N}_0(c(s_0)))=0, \quad\forall s\in[s_0-\delta,s_0+\delta],
\]
with
\[
V(\mathcal{N}_0(c(s_0)):=\bigcap\big{\{}i_{U*}H_1(U,\R)\subset H_1(\T,\R)|U \text{is an open neighborhood of }\mathcal{N}_0(c(s_0))\big{\}}
\]
and
\[
\mathcal{N}_0(c(s_0)):=\mathcal{N}(c(s_0))\bigcap\{t=0\}.
\]
\end{lem}
This Lemma is trivial because $\widetilde{\mathcal{N}}(c)$ is closed and due to Lemma \ref{lem}, we can always cover ${\mathcal{N}}_0(c)$ with finitely many disconnected open intervals $\{U_i=(a_i,b_i)\subset\T\}_{i=1}^k$. So we can find $[\xi(x)]=c(s)-c(s_0)$ such that supp$\xi(x)\bigcap(\bigcup_{i=1}^kU_i)=\emptyset$. The following variational principle will help us to construct a connecting orbit which passes by $\widetilde{\mathcal{A}}(c(s))$ and $\widetilde{\mathcal{A}}(c(s_0))$.
\begin{lem}\cite{CY1, CY2}
For the modified Lagrangian
\[
L_{\eta,\mu,\rho}(x,v,t):=L(x,v,t)-\langle\eta(x)+\rho(t)\mu(x),v\rangle,\quad(x,v,t)\in TM\times\R,
\]
where $[\eta]=c(s_0)$, $[\mu]=c(s)-c(s_0)$ and $\rho(t):\R\rightarrow\R$ is a smooth transitional function with $\rho(t)=1$ for $t\in[\delta,\infty)$ and $\rho(t)=0$ for $t\in(-\infty,0]$, we can define an action function by
\be
h_{\eta,\mu,\rho}^{T_0,T_1}(m,m')=\inf_{\substack{\gamma(-T_0)=m\\ \gamma(T_1)=m'}}\int_{-T_0}^{T_1}L_{\eta,\mu,\rho}(\gamma(t),\dot{\gamma}(t),t)dt+T_0\alpha(c(s_0))+T_1\alpha(c(s))
\ee
for integers $T_0$, $T_1$ and
\be\label{pseudo}
h_{\eta,\mu,\rho}^\infty(m,m'):=\liminf_{T_0,T_1\rightarrow\infty}h_{\eta,\mu,\rho}^{T_0,T_1}(m,m'),\quad m,m'\in M.
\ee
Let $\{T_0^i\}_{i\in\mathbb{N}}$ and $\{T_1^i\}_{i\in\mathbb{N}}$ be the sequence of positive numbers such that $T_j^i\rightarrow\infty$ $(j=0,1)$ as $i\rightarrow\infty$ and satisfies
\[
\lim_{i\rightarrow\infty}h_{\eta,\mu,\rho}^{T_0^i,T_1^i}(m,m')=h_{\eta,\mu,\rho}^{\infty}(m,m'),
\]
we can find $\gamma_i(t,m,m'):[-T_0^i,T_1^i]\rightarrow M$ being the minimizer connecting $m_0$ and $m_1$ such that
\[
h_{\eta,\mu,\rho}^{T_0^i,T_1^i}(m_0,m_1)=\int_{-T_0^i}^{T_1^i}L_{\eta,\mu,\rho}(\gamma_i(t),\dot{\gamma}_i(t),t)dt+T_0^i\alpha(c(s_0))+T_1^i\alpha(c(s)).
\]
Then the set $\{\gamma_i\}$ is pre-compact in $C^1(\mathbb{R},M)$. Let $\gamma:\mathbb{R}\rightarrow M$ be an accumulation orbit of $\{\gamma_i\}$, then $\forall s,\tau\geq1$
\be\label{pseudo curve}
\int _{-s}^\tau L_{\eta,\mu,\rho}(d{\gamma},t) dt&=&\inf_{\substack{s_1-s\in\mathbb{Z},\tau_1-\tau\in\mathbb{Z}\nonumber\\
s_1,\tau_1\geq1\\
\gamma^*(-s_1)=\gamma(-s)\\
\gamma^*(\tau_1)=\gamma(\tau)}}\int_{-s_1}^{\tau_1}L_{\eta,\mu,\rho}(d\gamma^*(t),t)dt\\
& &+(s_1-s)\alpha(c(s_0)) +(\tau_1-\tau)\alpha(c(s)),
\ee
If we denote by $\mathscr{C}_{\eta,\mu,\rho}$ the set of all the minimizers of (\ref{pseudo}), then any orbit $\gamma(t):\R\rightarrow M$ in it conforms to the Euler Lagrange equation 
\be\label{e-l}
\frac{d}{dt}L_v(\gamma(t),\dot{\gamma}(t),t)=L_x(\gamma(t),\dot{\gamma}(t),t),\quad\forall t\in\R,
\ee
as long as we take $s$ sufficiently close to $s_0$. Moreover, $\gamma$ is a heteroclinic orbit connecting $\mathcal{A}(c(s_0))$ and $\mathcal{A}(c(s))$. 
\end{lem}
\begin{proof}
This is a direct citation of conclusions in \cite{CY1,CY2}, so here we just give a sketch of the strategy: Recall that $L_{\eta,\mu,\rho}$ is also superlinear and positively definite, because we just add a linear term of $v$ to the original $L(x,v,t)$. So we can get the compactness of $\{\gamma_i\}$ from the superlinearity. In turns we get the existence of accumulation orbits and $\mathscr{C}_{\eta,\mu,\rho}\neq\emptyset$. Moreover, $\mathscr{C}_{\eta,\mu,\rho}$ is a upper semi-continuous set valued function of the additional term $\{\mu,\rho\}$, which is due to the pre-compactness as well. Recall that $\mathcal{N}_0(c(s_0))\subset\bigcup_{i=1}^kU_i$ and supp$\mu(x)\bigcap(\bigcup_{i=1}^kU_i)=\emptyset$, there must exist a small time $0<\delta<1$
such that $\mathcal{N}(c(s_0))\bigcap\{t\in[0,\delta]\}\subset\bigcup_{i=1}^kU_i$, then we use the upper semi-continuity of $\mathscr{C}_{\eta,\mu,\rho}$, for sufficiently small $|c(s)-c(s_0)|$, $\mathscr{C}_{\eta,\mu,\rho}\bigcap\{t\in[0,\delta]\}\subset\bigcup_{i=1}^kU_i$ as well. This is because $\mathscr{C}_{\eta,0,0}=\mathcal{N}(c(s_0))$. On the other side, only for $t\in[0,\delta]$ we have $L_{\eta,\mu,\rho}\neq L$, whereas $\forall\gamma\in\mathscr{C}_{\eta,\mu,\rho}$, supp$\mu(x)\bigcap(\bigcup_{i=1}^kU_i)=\emptyset$ makes $\gamma|_{[0,\delta]}$ conforms to the same Euler Lagrange equation as (\ref{e-l}).
\end{proof}
\begin{rmk}
This locally connecting mechanism was initially used to solve the {\it a priori} unstable Arnold diffusion problems. Similar skill was also revealed independently in \cite{B}.

Corollary \ref{cor} actually tells us that, in geometrical optical meaning, adjacent caustics will form a closed, but instable region (so called `caustic tube' in \cite{L}).
\end{rmk}

At last, Iet's make a comment about the shape of $\beta(h)$ function at $\{h=0\}$. Indeed we can see that the positive definiteness of $L(x,v,t)$ collpases at the level set $\{v=0\}$, so the $\beta(h)$ function have a quite different shape from the one of general mechanical systems. For $c\in H^1(\T,\R)$ sufficiently close to 0,
\be
\alpha(c)&=&-\inf_{\mu\in\mathfrak{M}_L}\int L-c\; d\mu\nonumber\\
&=&\sup_{\mu\in\mathfrak{M}_c}\int c-L\; d\mu\nonumber\\
&\geq&\frac{1}{T}\int_0^{T} \frac{2\sqrt 2}{3}c^{3/2}+\cO(c^{5/2})\; dt\quad\text{by taking $v=\sqrt{2c_n}$}\nonumber\\
&=&\frac{2\sqrt 2}{3}c^{3/2}+\cO(c^{5/2})\nonumber.
\ee
On the other side, for $h\in H_1(\T,\R)$ sufficiently close to 0,
\be
\beta(h)&=&\max_c\{\langle c,h\rangle-\alpha(c)\}\nonumber\\
&\leq&\max_c\{\langle c,h\rangle-\frac{2\sqrt2}{3}c^{3/2}+\cO(c^{5/2})\}\nonumber\\
&\leq&\frac{h^3}{6}+\cO(h^5)\nonumber.
\ee
It's trivial that $\beta(h)\geq0$, so we get the degeneracy of $\beta(h)$ at zero by
\[
\lim_{h\rightarrow0}\frac{\beta(h)}{h^{2+\varpi}}=0,\quad \forall\varpi\in(0,1).
\]
\textr{Because $\alpha(c)$ is $C^1$ smooth, so $\alpha'(0)=h=0$. Then we use the continuity of $\alpha'(c)$ get that as $c\rightarrow0$, $h_c\rightarrow0$ as well.}
\begin{figure}
\begin{center}
\includegraphics[width=12cm]{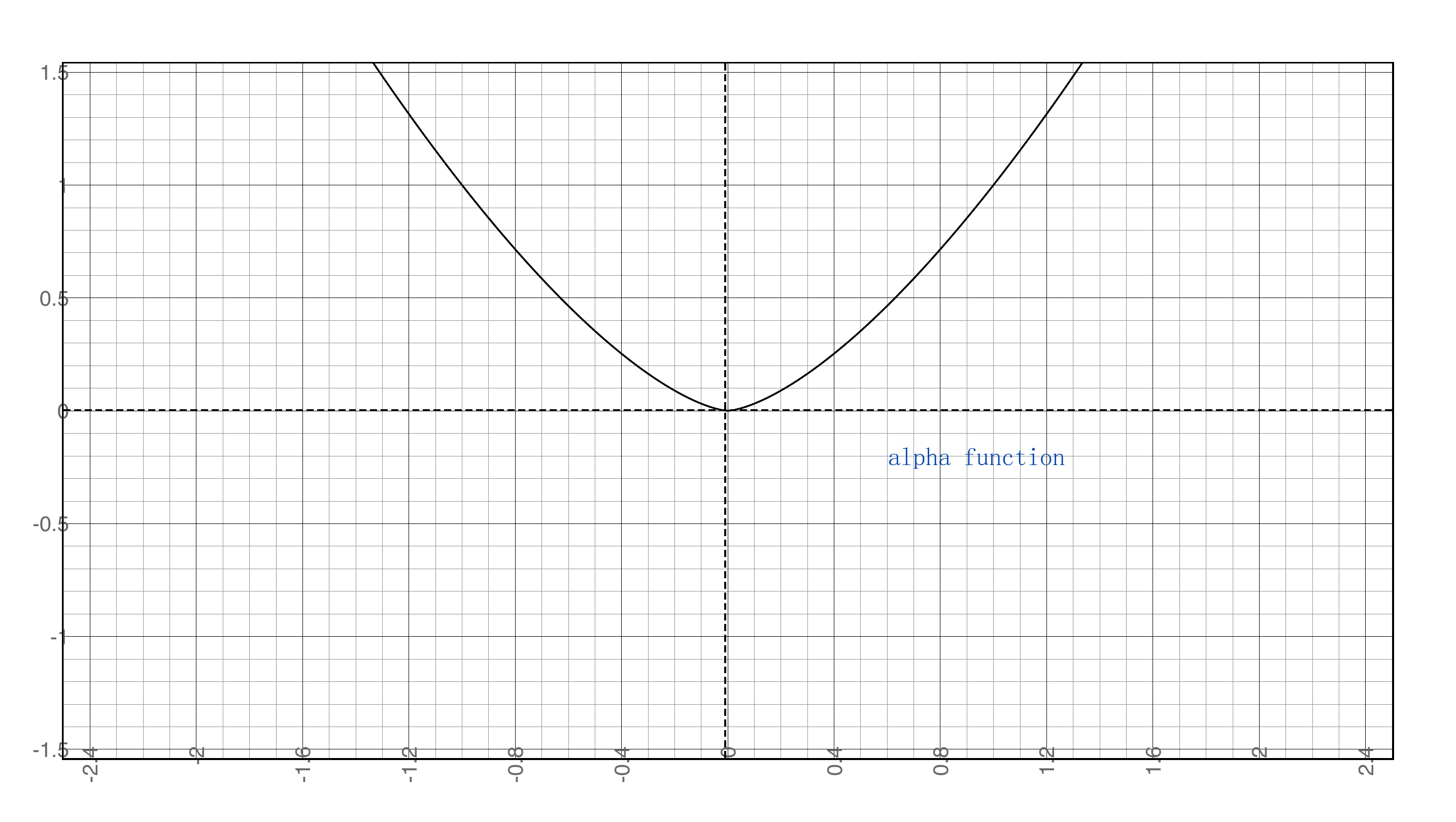}
\caption{ }
\label{fig4}
\end{center}
\end{figure}
\begin{ex}
For the circular billiard map, we can get
\[
\alpha(c)=4\rho c\arcsin\sqrt{\frac{c}{8\rho^2}}-16\rho^3\arcsin\sqrt{\frac{c}{8\rho^2}}+2\rho\sqrt{8\rho^2c-c^2}
\]
(see figure \ref{fig4}), then
 \[
 \lim_{h\rightarrow0}\beta(h)/h^3=1/6
 \]
where $\rho$ is a constant.
\end{ex}
\section*{Acknowledgement}
The author thanks Prof. Kostya Khanin for offering the postdoc position in the University of Toronto, in which this work is finished. I also warmly thank Prof. Vadim Kaloshin and Prof. Ke Zhang for communications on the convex billiard problems and many useful suggestions.

\end{document}